\documentclass[a4paper,12pt]{amsart}
\usepackage{amssymb, amsmath, graphicx}
\usepackage[curve]{xypic}
\usepackage{enumerate}
\usepackage{tikz-cd}  
\usepackage[indexonlyfirst,sort=use]{glossaries}
\usepackage{url}
\DeclareRobustCommand{\gobblefour}[4]{}
\newcommand*{\SkipTocEntry}{\addtocontents{toc}{\gobblefour}}

\providecommand{\abs}[1]{\lvert#1\rvert}

\providecommand{\Ker}{\textnormal{Ker}}

\providecommand{\Ker}{\textnormal{Ker}}

\providecommand{\Tor}{\textnormal{Tor}}

\providecommand{\N}{\mathbb{N}}
\providecommand{\Z}{\mathbb{Z}}
\providecommand{\Q}{\mathbb{Q}}
\providecommand{\R}{\mathbb{R}}
\providecommand{\C}{\mathbb{C}}

\tikzset{
	curvarr/.style={
		to path={ -- ([xshift=2ex]\tikztostart.east)
			|- (#1) [near end]\tikztonodes
			-| ([xshift=-2ex]\tikztotarget.west)
			-- (\tikztotarget)}
	}
}

\makeatletter
\newcommand\tint{\mathop{\mathpalette\tb@int{t}}\!\int}
\newcommand\bint{\mathop{\mathpalette\tb@int{b}}\!\int}
\newcommand\tb@int[2]{%
  \sbox\z@{$\m@th#1\int$}%
  \if#2t%
    \rlap{\hbox to\wd\z@{%
      \hfil
      \vrule width .35em height \dimexpr\ht\z@+1.4pt\relax depth -\dimexpr\ht\z@+1pt\relax
      \kern.05em % a small correction on the top
    }}
  \else
    \rlap{\hbox to\wd\z@{%
      \vrule width .35em height -\dimexpr\dp\z@+1pt\relax depth \dimexpr\dp\z@+1.4pt\relax
      \hfil
    }}
  \fi
}
\makeatother

\allowdisplaybreaks

\begin{document}

\title[Introduction to the BSD Conjecture]{Introduction to the Birch and Swinnerton-Dyer Conjecture}
\author{Fabio Ferrari Ruffino}
\address{Departamento de Matem\'atica - Universidade Federal de S\~ao Carlos - Rod.\ Washington Lu\'is, Km 235 - C.P.\ 676 - 13565-905 S\~ao Carlos, SP, Brasil}
\email{ferrariruffino@ufscar.br}

\begin{abstract}
We give a brief and partially informal introduction to the Birch and Swinner\-ton-Dyer conjecture. After a short review on basic Diophantine equations, we define the natural abelian group structure on an elliptic curve, leading to the definition of the corresponding algebraic rank. Afterwards, we sketch how to define the analytic rank of an elliptic curve and we state the conjecture, that is, the equality of these two ranks.
\end{abstract}

\maketitle

\newtheorem{Theorem}{Theorem}[section]
\newtheorem{Lemma}[Theorem]{Lemma}
\newtheorem{Prop}[Theorem]{Proposition}
\newtheorem{Corollary}[Theorem]{Corollary}
\newtheorem{ThmDef}[Theorem]{Theorem - Definition}
\newtheorem*{BSDConj}{Birch and Swinnerton-Dyer Conjecture}

\theoremstyle{definition}
\newtheorem{Rmk}[Theorem]{Remark}
\newtheorem{Rmks}[Theorem]{Remarks}
\newtheorem{Def}[Theorem]{Definition}
\newtheorem{Not}[Theorem]{Notation}

%%%%%%%%%%%%%%%%%%%%%%%%%%%%%%%%%%%%%%%%%%%%%%%%%%%%%%%%%%%%%%%%%%%%%%%%%%%%%%%%%%%%%%%%%%%%%%%

\tableofcontents

%%%%%%%%%%%%%%%%%%%%%%%%%%%%%%%%%%%%%%%%%%%%%%%%%%%%%%%%%%%%%%%%%%%%%%%%%%%%%%%%%%%%%%%%%%%%%%%

\section{Introduction}

The Birch and Swinnerton-Dyer Conjecture concerns rational points on smooth cubic plane curves, thus it is part of the research on third-degree Diophantine equations in two variables (see \cite{Johnson} and \cite[Section 6.6]{ST}). In general, given a curve of this kind, we do not know yet how to establish whether it contains rational points or not, thus we have to assume that there exists at least one. Starting from this hypothesis, it is possible to construct a natural abelian group structure on the set of rational points, that turns out to be finitely generated. The rank of this group is called \emph{algebraic rank} of the cubic. Of course, knowing the value of this rank is very relevant to describe all of the rational points of the curve, therefore it would be important to have a concrete way to compute it. Unfortunately, this is a hard problem in general (there are plenty of open problems about this topic; for example, we neither know if the set of possible algebraic ranks is bounded above or not).

Since the construction of the abelian group is birationally invariant, it actually applies to any projective curve that is birationally equivalent to a smooth rational plane cubic (equivalently, to a smooth projective curve of genus $1$ defined over $\Q$). This extended family includes certain quartic curves that naturally arise as arguments of elliptic integrals. That's way the curves of this family, with a fixed rational point, are called \emph{elliptic curves}. Typically, an elliptic curve is represented concretely as a pair $(\mathfrak{C}, O)$, where $\mathfrak{C}$ is a smooth rational plane cubic in the Weierstrass form $y^{2} = x^{3} + bx + c$, with $4b^{3} + 27c^{2} \neq 0$, and $O$ is a specified rational point.

With the help of complex analysis, we can define another notion of rank, that is easier to compute in various meaningful cases. In fact, for every prime $p$, we consider the reduction modulo $p$ of an elliptic curve. Such reduction is smooth too, except for a finite number of cases that can be easily overcome. For every $p$, we compute the difference between the number of points of the reduced curve and the expected average cardinality. We get a list of integral numbers $\epsilon_{p}$ that can be successfully organized in a Dirichelet Series, called \emph{L-function} of the elliptic curve, of the following form:
	\[L_{\mathfrak{C}}(s) := \prod_{p \nmid \Delta} \Bigl(1 - \frac{\epsilon_{p}}{p^{s}} + \frac{1}{p^{2s-1}}\Bigr)^{-1} \cdot \prod_{p \mid \Delta} \Bigl(1 - \frac{\epsilon_{p}}{p^{s}}\Bigr)^{-1} = \sum_{n=1}^{+\infty} \frac{\epsilon_{n}}{n^{s}}
\]
Here $\Delta$ is the discriminant of the curve, $p \mid \Delta$ means that $p$ is a prime divisor of $\Delta$, and $\epsilon_{n}$ is a suitable extension of the values $\epsilon_{p}$ to not necessarily prime indices. This series can be analytically extended to the whole complex plane. We define the \emph{analytic rank} of the curve as the order of this holomorphic function in the point $1$.

Some ``experimental'' evidences, obtained through computations with the most powerful computers of that time, together with a heuristic theoretical argument, led the two young researchers Birch and Swinnerton-Dyer, working at Cambridge University around 1958, to conjecture that the algebraic and the analytic rank of any elliptic curve coincide. We do not know yet whether they were right or not, but, with more modern computers, we found stronger evidences about this affirmation. Furthermore, if it is true, then it allows us to know the algebraic rank of an elliptic curve by computing its analytic rank, which is easier to calculate in some meaningful situations. Beyond this advantage, some historical problems in number theory, like the \emph{congruent number problem}, would finally find a solution. For these reasons, the Birch and Swinnerton-Dyer Conjecture was included in the list of the Millennium Prize Problems by the Clay Mathematics Institute in 2000, confirming its strong relevance in contemporary mathematics.

%%%%%%%%%%%%%%%%%%%%%%%%%%%%%%%%%%%%%%%%%%%%%%%%%%%%%%%%%%%%%%%%%%%%%%%%%%%%%%%%%%%%%%%%%%%%%%%

\section{Diophantine Equations}

A \emph{Diophantine equation} is a polynomial equation with integral (equivalently, rational) coefficients, whose solutions are by definition integral or rational, depending on the context. Typically, we are interested in knowing both the set of rational solutions and its subset of integral solutions. Therefore, we have four meaningful problems:
\begin{enumerate}
	\item[(P1)] establishing whether there exists a rational solution or not
	\item[(P2)] finding all of the rational solutions, provided that they exist
	\item[(P3)] establishing whether there exists an integral solution or not
	\item[(P4)] finding all of the integral solutions, provided that they exist
\end{enumerate}
When there is only one variable, each of these problems can be solved quite easily in principle. In fact, the equation is of the following form:
\begin{equation}\label{DiophEqOneV}
	a_{n}x^{n} + a_{n-1}x^{n-1} + \cdots + a_{1}x + a_{0} = 0, \quad a_{0}, \ldots, a_{n} \in \Z, \; a_{n} \neq 0
\end{equation}
We can assume without loss of generality that $n > 0$ and $a_{0} \neq 0$. If $\frac{k}{h}$ is a rational solution with $k$ and $h$ coprime, then we have
	\[a_{n}k^{n} + a_{n-1}k^{n-1}h + \cdots + a_{1}kh^{n-1} + a_{0}h^{n} = 0.
\]
It follows that $k \mid a_{0}h^{n}$, thus $k \mid a_{0}$, and $h \mid a_{n}k^{n}$, thus $h \mid a_{n}$. Since $a_{0}$ and $a_{n}$ have finitely many divisors, we have a finite set of possible rational solutions, some of which are integral. Hence, we analyse each element of this set and we verify whether it satisfies the equation or not.

The situation is much more complicated with two or more variables, since, in this case, we have no general method at all. From now on we concentrate on the case of two variables, considering equations of low degree.

\SkipTocEntry \subsection{Linear Equations}

A first degree Diophantine equation is called \emph{linear} (even if the constant term does not vanish). We get an equation of the following form:
\begin{equation}\label{LinDiophEq}
	ax + by + c = 0, \quad a, b, c \in \Z, \; (a, b) \neq (0, 0)
\end{equation}
Equivalently, we seek the rational and integral points of the line $ax + by + c = 0$ in $\R^{2}$. If $b \neq 0$, then the real solutions are the pairs of the form $\bigl(t, -\frac{at+c}{b}\bigr)$; clearly, both entries are rational if and only if $t \in \Q$. The analogous argument holds if $a \neq 0$. This solves problems (P1) and (P2). About (P3) and (P4), we call $d$ the greatest common divisor of $a$ and $b$, which can be computed through the Euclidean division algorithm. If there exists an integral solution, then $d \mid c$. Conversely, let us suppose that $d \mid c$. By the Bezout's identity, there exist $k, h \in \Z$ such that $ak + bh = d$; we can compute a pair $(k, h)$ with this property through back substitution in the Euclidean algorithm. It immediately follows that $(-k\frac{c}{d}, -h\frac{c}{d})$ is an integral solution. In order to find all of the other ones, we have to add the generic solution of the homogeneous equation $ax + by = 0$, that is, $(n\frac{b}{d}, - n\frac{a}{d})$ for any $n \in \Z$. Summarizing:
\begin{enumerate}
	\item[(P1)] There always exist infinitely many rational solutions.
	\item[(P2)] The rational solutions are the pairs of the form $\bigl(t, -\frac{at+c}{b}\bigr)$ or $\bigl(-\frac{bt+c}{a}, t\bigr)$ for any $t \in \Q$.
	\item[(P3)] There exists an integral solution if and only if $d \mid c$, where $d := \textnormal{GCD}(a, b)$ can be computed trough the Euclidean division algorithm.
	\item[(P4)] The integral solutions are the pairs of the form $\bigl(\frac{nb-kc}{d}, -\frac{na+hc}{d}\bigr)$ for any $n \in \Z$, where $k$ and $h$ satisfy the identity $ak + bh = d$; such a pair can be computed through back substitution in the Euclidean algorithm.
\end{enumerate}

\SkipTocEntry \subsection{Equations of Degree $2$}

From now on we focus on problems (P1) and (P2), since they are directly related to the main topic of this review (that is, the Birch and Swinnerton-Dyer Conjecture).

Solving Diophantine equations of degree $2$ in two variables---equivalently, finding the rational points of a given conic---is not as easy as in the linear case, but we have a general algorithm anyway. First of all, not every conic with integral coefficients admits rational points. For example, let us consider the circle $x^{2} + y^{2} = 3$ and, by contradiction, let $\bigl(\frac{k}{l}, \frac{h}{l}\bigr)$ be a rational point. We can suppose without loss of generality that $k$, $h$, and $l$ have no common factors. We get $k^{2} + h^{2} = 3l^{2}$. If $3 \mid k$, then $3 \mid (3l^{2} - k^{2}) = h^{2}$, thus $3 \mid h$ too. It follows that $9 \mid (k^{2} + h^{2}) = 3l^{2}$, thus $3 \mid l$, while $k$, $h$, and $l$ have no common factors by hypothesis. This shows that $3 \nmid k$ and, by the analogous argument, $3 \nmid h$. Therefore, $k = 3r \pm 1$ and $h = 3s \pm 1$, that implies $k^{2} + h^{2} = 3t + 2 \neq 3l^{2}$, a contradiction. Hence, each of the problems raised above is non-trivial, including the first one.

\subsubsection{Problem (P1)} We start from the equation
\begin{equation}\label{DiophConic}
	ax^{2} + bxy + cy^{2} + dx + ey + f = 0, \quad a, \ldots, f \in \Z, \; (a, b, c) \neq (0, 0, 0)
\end{equation}
We suppose that the conic is irreducible, since otherwise the analysis of rational and integral points can be realised without difficulties.\footnote{The only issue to be considered is that a conic with rational coefficients can factorize in two lines with some irrational coefficients (for example, $x^{2} - 2 = 0$ or $2x^{2} - 3y^{2} = 0$). Nevertheless, even in these cases, we can easily solve problems (P1)--(P4).} We consider the corresponding homogeneous equation, that is, we replace $x$ and $y$ respectively by $\frac{x}{u}$ and $\frac{y}{u}$. We get:
\begin{equation}\label{DiophConicHom}
	ax^{2} + bxy + cy^{2} + dxu + eyu + fu^{2} = 0
\end{equation}

\begin{Lemma} The following statements are equivalent:
\begin{enumerate}
	\item[(i)] Equation \eqref{DiophConic} admits a rational solution.
	\item[(ii)] Equation \eqref{DiophConicHom} admits a non-vanishing integral solution.
	\item[(iii)] Equation \eqref{DiophConicHom} admits a non-vanishing rational solution.
\end{enumerate}
If these conditions hold, then \eqref{DiophConicHom} admits an integral solution $(k, h, l)$ such that $l \neq 0$.
\end{Lemma}
\begin{proof} (i) $\Rightarrow$ (ii) If $\bigl(\frac{k}{l}, \frac{h}{l}\bigr)$ is a rational solution of \eqref{DiophConic}, then $(k, h, l)$ is an integral solution of \eqref{DiophConicHom} such that $l \neq 0$. (ii) $\Rightarrow$ (i) If $(k, h, l)$ is an integral solution of \eqref{DiophConicHom} and $l \neq 0$, then $\bigl(\frac{k}{l}, \frac{h}{l}\bigr)$ is a rational solution of \eqref{DiophConic}. If $l = 0$, then we suppose $k \neq 0$, the argument being analogous if $h \neq 0$. We divide \eqref{DiophConicHom} by $x^{2}$ and we set $\bar{y} := \frac{y}{x}$ and $\bar{u} := \frac{u}{x}$:
\begin{equation}\label{DiophConicBar}\tag{$\star$}
	a + b\bar{y} + c\bar{y}^{2} + d\bar{u} + e\bar{y}\bar{u} + f\bar{u}^{2} = 0
\end{equation}
The pair $\bigl(\frac{h}{k}, 0\bigr)$ is a rational point of \eqref{DiophConicBar}. Solving problem (P2), we will see that, if a conic admits a rational point, then it admits infinitely many. Since there are at most two points of \eqref{DiophConicBar} with $\bar{u} = 0$, there exists a rational point $\bigl(\frac{r}{t}, \frac{s}{t}\bigr)$ with $s \neq 0$. By multiplying by $t^{2}$, we deduce that $(t, r, s)$ is an integral point of \eqref{DiophConicHom} with $s \neq 0$. Hence, $\bigl(\frac{t}{s}, \frac{r}{s}\bigr)$ is a rational solution of \eqref{DiophConic}. (ii) $\Rightarrow$ (iii) Obvious. (iii) $\Rightarrow$ (ii) If $\bigl(\frac{k}{l}, \frac{h}{l}, \frac{r}{l}\bigr)$ is a non-vanishing rational solution, then $(k, h, r)$ is a non-vanishing integral solution.
\end{proof}
If $f = 0$, then $(0, 0, 1)$ is an integral point of \eqref{DiophConic}, thus we suppose $f \neq 0$ and we apply the following algorithm:
\begin{itemize}
	\item We write \eqref{DiophConicHom} in the form
		\[\textstyle f\bigl(u + \frac{d}{2f}x + \frac{e}{2f}y\bigr)^{2} - \frac{d^{2}}{4f^{2}}x^{2} - \frac{e^{2}}{4f^{2}}y^{2} - \frac{de}{2f^{2}}xy + ax^{2} + bxy + cy^{2} = 0.
	\]
	We set $\hat{u} := u + \frac{d}{2f}x + \frac{e}{2f}y$ and we get $f\hat{u}^{2} + \alpha x^{2} + \beta xy + \gamma y^{2} = 0$.
	\item If $\alpha = \gamma = 0$, then $y = -\frac{f\hat{u}^{2}}{\beta x}$ or $x = -\frac{f\hat{u}^{2}}{\beta y}$. In this case, $\bigl(x, -\frac{f\hat{u}^{2}}{\beta x}, 1\bigr)$ is a rational point if and only if $x \in \Q$, and $\bigl(-\frac{f\hat{u}^{2}}{\beta y}, y, 1\bigr)$ is a rational point if and only if $y \in \Q$. Thus, we suppose $\alpha \neq 0$, the argument being analogous if $\gamma \neq 0$.
	\item We get $f\hat{u}^{2} + \alpha(x + \frac{\beta}{2\alpha}y)^{2} - \frac{\beta^{2}}{4\alpha^{2}}y^{2} - \gamma y^{2}$. We set $\hat{x} := x + \frac{\beta}{2\alpha}y$ and we get $f\hat{u}^{2} + \alpha\hat{x}^{2} + \eta y^{2} = 0$. Since the coordinate change $(x, y, u) \mapsto (\hat{x}, y, \hat{u})$ and its inverse are rational, \eqref{DiophConicHom} admits a non vanishing integral solution if and only if $f\hat{u}^{2} + \alpha\hat{x}^{2} + \eta y^{2} = 0$ does. Thus, we reduced the original problem to a homogeneous conic of the form
\begin{equation}\label{DiophConicHomRed}
	ax^{2} + by^{2} + cu^{2} = 0,
\end{equation}
where $a, b, c \in \Z$.
	\item If $abc = 0$, then the conic \eqref{DiophConicHomRed} is reducible, hence we suppose $abc \neq 0$. Moreover, if $a$, $b$, and $c$ have the same sign, then there are no non-vanishing real solutions; thus, we suppose that this does not happen. Lastly, we suppose that $abc$ is square-free. Indeed, we can trivially assume that $a$, $b$, and $c$ have no common factors; also, if for example $a = p^{2}a'$, then we set $\tilde{x} := px$ and we get the equivalent equation $a' \tilde{x}^{2} + b y^{2} + c z = 0$, eliminating the factor $p^{2}$; furthermore, if for example $p \mid a$, $p \mid b$, and $p \nmid c$, then we set $\tilde{z} := \frac{z}{p}$ and we get the equivalent equation $\frac{a}{p} x^{2} + \frac{b}{p} y^{2} + cp \tilde{z}^{2} = 0$, replacing the factor $p^{2}$ in $abc$ with $p$. The following theorem, whose proof can be found in \cite[Theorem 85]{Kumar}, completes the algorithm.
\end{itemize}
\begin{Theorem}[Legendre]\label{LegThm} The homogeneous conic \eqref{DiophConicHomRed}, where:
\begin{itemize}
	\item $abc \neq 0$;
	\item $a$, $b$, and $c$ do not have the same sign; and
	\item $abc$ is square-free
\end{itemize}
admits a non-vanishing integral solution if and only if there exist $k, h, l \in \Z$ such that $-ab \equiv_{c} k^{2}$, $-ac \equiv_{b} h^{2}$, and $-bc \equiv_{a} l^{2}$.
\end{Theorem}
For example, we have seen that the circle $x^{2} + y^{2} - 3 = 0$ does not have rational points. Indeed, with the notation of theorem \ref{LegThm}, we have $a = b = 1$ and $c = -3$. The equivalences $3 \equiv_{1} h^{2}$ and $3 \equiv_{1} l^{2}$ are trivially satisfied for any $h$ and $l$, but the equivalence $-1 \equiv_{3} k^{2}$---that is, $2$ is a square in $\Z_{3}$---is false. On the contrary, the circle $x^{2} + y^{2} - 5 = 0$ admits rational points---for example, $(\pm 2, \pm 1)$ and $(\pm 1, \pm 2)$. Indeed, we have $a = b = 1$ and $c = -5$, and the equivalence $-1 \equiv_{5} k^{2}$---that is, $4$ is a square in $\Z_{5}$---is true.

\subsubsection{Problem (P2)} Now we consider a conic $\mathfrak{C}$ with a fixed rational point $P$ and we show how to find all of the other ones. We proceed as follows:
\begin{itemize}
	\item We fix a line $\mathfrak{L}$ with rational coefficients and parallel to the the tangent line to $\mathfrak{C}$ in $P$, but different from it (that is, $P \notin \mathfrak{L}$).
	\item We get the projection $\pi \colon \mathfrak{C} \setminus \{P\} \to \mathfrak{L}$ defined as follows: for every point $Q \in \mathfrak{C} \setminus \{P\}$, we consider the line passing trough $P$ and $Q$ and we call $R$ the intersection point between this line and $\mathfrak{L}$; we set $\pi(Q) := R$.
	\item This projection is injective and it is possible to prove that the complement of the image contains at most two points. Its inverse is defined as follows: for every point $R$ in the image of $\pi$, we consider the line through $P$ and $R$, that intersects the conic in $P$ itself and in a point $Q$ different from $P$; we have $Q = \pi^{-1}(R)$.
	\item Since $P$ and the coefficients of $\mathfrak{L}$ are rational, the point $Q$ is rational if and only if $\pi(Q)$ is rational. Thus, we solve the problem by computing the rational points of $\mathfrak{L}$.
\end{itemize}
Let us show an easy example. The conic $\mathfrak{C}$ is the unit circle---that is, $x^{2} + y^{2} = 1$---and we fix the rational point $P := (-1, 0)$. The line $\mathfrak{L}$ is $x = 0$, that is parallel to the tangent line to $\mathfrak{C}$ in $P$. We get the projection $\pi \colon \mathfrak{C} \setminus \{P\} \to \mathfrak{L}$ that sends the point $Q = (\cos\theta, \sin\theta)$ to the point $R = (0, \tan\frac{\theta}{2})$, where $-\pi < \theta < \pi$, as shown in figure \ref{ProjCircleLine}. The inverse projection sends $R = (0, t)$ to $Q = \bigl( \frac{1-t^{2}}{1+t^{2}}, \frac{2t}{1+t^{2}} \bigr)$. Since $(0, t)$ is rational if and only if $t \in \Q$, it follows that the rational points of the unit circle are $(-1, 0)$ and the points of the form $\bigl( \frac{1-t^{2}}{1+t^{2}}, \frac{2t}{1+t^{2}} \bigr)$ with $t \in \Q$.

\begin{figure}
\begin{center}
\begin{tikzpicture}
\draw[->] (-70pt, 0pt) -- (70pt, 0pt);
\draw[->] (0pt, -70pt) -- (0pt, 70pt);

\draw[green] (0pt, 0pt) -- (45:60pt);
\draw[green] (-60pt, 0pt) -- (45:60pt);

\draw[green] (15pt, 0pt) arc (0:45:15pt);
\draw[green] (-40pt, 0pt) arc (0:22.5:20pt);

\draw[red, thick] (0pt, -70pt) -- (0pt, 69pt);
\draw[blue, thick] (0pt, 0pt) circle (60pt);

\filldraw[green] (45:60pt) circle (1.5pt);
\filldraw[green] (0pt, 24.85pt) circle (1.5pt);
\filldraw[thick, blue, fill=white] (-60pt, 0pt) circle (1.5pt);

\draw[green] (0pt, 0pt) (45:68.5pt) node {$Q$};
\draw[green] (-7.5pt, 32.5pt) node {$R$};
\draw[green] (22.5:22.5pt) node {$\theta$};
\draw[green] (-60pt, 0pt) + (11.25:27.5pt) node {$\scriptscriptstyle \frac{\theta}{2}$};
\draw[blue] (-67.5pt, 7.5pt) node {$P$};
\end{tikzpicture}
\end{center}
\caption{} \label{ProjCircleLine}
\end{figure}

\SkipTocEntry \subsection{Equations of Degree $3$}

Let us consider a cubic equation in two variables with integral coefficients:
\begin{equation}\label{DiophCubic}
\begin{split}
	ax^{3} + bx^{2}y + cxy^{2} + dy^{3} + ex^{2} + fxy + gy^{2} \, & + hx + iy + j = 0, \\
	& a, \ldots, j \in \Z, \; (a, b, c, d) \neq (0, 0, 0, 0)
\end{split}
\end{equation}
First of all, we need to consider the \emph{projective closure} of this curve, thus we briefly summarise what does this mean.

\subsubsection{Mini-review about the projective plane}

In the affine plane $\R^{2}$, two distinct points induce a unique line, while two distinct lines induce a unique point provided that they are not parallel. The projective plane can be thought of as a completion of $\R^{2}$, in such a way that this asymmetry is eliminated. In fact, we add a ``point at infinity'' for each possible direction of a line, so that two parallel lines intersect in the unique point at infinity that represents their common direction. We can construct a model of the projective plane as follows: we identify $\R^{2}$ with the open unit ball in itself; then, we add the boundary of the ball---that is, the unit circle---and we identify two points in the circle when they are opposite with respect to the origin. In this way, a pair of antipodal points in the circle is a point at infinity. It follows that lines in $\R^{2}$ are represented as in figure \ref{ProjPlane}, where two lines are parallel if and only if they share the same point at infinity. In particular, the diameters of the disc correspond to the lines passing trough the origin. The set of points at infinity---that is, the unit circle with antipodal points identified---turns out to be a line too, that is called \emph{line at infinity}.

\begin{figure}
\begin{center}
\begin{tikzpicture}
\filldraw[thick, blue, fill=cyan!30!white] (0pt, 0pt) circle(50pt);

\draw[thick, magenta] (225:50pt) -- (45:50pt);
\draw[thick, magenta] (225:50pt) .. controls (135:20pt) .. (45:50pt);
\draw[thick, magenta] (225:50pt) .. controls (315:20pt) .. (45:50pt);
\filldraw[magenta] (225:50pt) circle (1.5pt);
\filldraw[magenta] (45:50pt) circle (1.5pt);

\draw[thick, black] (135:50pt) -- (315:50pt);
\draw[thick, black] (135:50pt) .. controls (45:20pt) .. (315:50pt);
\draw[thick, black] (135:50pt) .. controls (225:20pt) .. (315:50pt);
\filldraw[black] (135:50pt) circle (1.5pt);
\filldraw[black] (315:50pt) circle (1.5pt);

\end{tikzpicture}
\end{center}
\caption{} \label{ProjPlane}
\end{figure}

Every algebraic curve in $\R^{2}$ admits its corresponding closure in the projective plane, that is formed by the curve itself and by its points at infinity. In order to compute such points, we have to consider the zeros of the highest-degree homogeneous component of the curve. For example, let us consider the irreducible conics, whose projective closures are shown in figure \ref{ProjConics}. Given the hyperbola $x^{2} - y^{2} - 1 = 0$, we get the homogeneous equation $x^{2} - y^{2} = 0$, whose solutions are the asymptotes $y = x$ and $y = -x$; thus, we get two points at infinity, that is, the directions of the asymptotes. Given the parabola $y - x^{2} = 0$, we get $x^{2} = 0$, thus we obtain the direction the $y$-axis counted twice; this means that there exists one point at infinity, in which the parabola is tangent to the line at infinity. Lastly, given the ellipse $x^{2} + y^{2} - 1 = 0$, we get $x^{2} + y^{2} = 0$, whose only solution is $(0, 0)$; in this case, we have no points at infinity. Analogous computations hold for any irreducible conic.

\begin{figure}
\begin{center}
\begin{tikzpicture}
\filldraw[thick, blue, fill=cyan!30!white] (0pt, 0pt) circle(50pt);

\draw[thick, magenta] (225:50pt) -- (45:50pt);
\draw[thick, black] (135:50pt) -- (315:50pt);
\draw[thick, orange] (135:50pt) .. controls (180:15pt) .. (225:50pt);
\draw[thick, orange] (45:50pt) .. controls (0:15pt) .. (315:50pt);

\filldraw[thick, magenta, fill=orange] (225:50pt) circle (1.5pt);
\filldraw[thick, magenta, fill=orange] (45:50pt) circle (1.5pt);
\filldraw[thick, black, fill=orange] (135:50pt) circle (1.5pt);
\filldraw[thick, black, fill=orange] (315:50pt) circle (1.5pt);

\begin{scope}[xshift=150pt]
\filldraw[thick, blue, fill=cyan!30!white] (0pt, 0pt) circle(50pt);
\draw[thick, orange] (20pt, 0pt) ellipse (30pt and 15pt);
\filldraw[orange] (50pt, 0pt) circle (1.5pt);
\filldraw[orange] (-50pt, 0pt) circle (1.5pt);
\end{scope}

\begin{scope}[xshift=300pt]
\filldraw[thick, blue, fill=cyan!30!white] (0pt, 0pt) circle(50pt);
\draw[thick, orange] (0pt, 0pt) ellipse (30pt and 15pt);
\end{scope}

\end{tikzpicture}
\end{center}
\caption{} \label{ProjConics}
\end{figure}

Thinking of $\R^{2}$ as an affine space, the corresponding automorphisms are affine transformations, that is, compositions between linear bijections and translations. We can suitably define the concept of automorphism of the projective plane, that we call \emph{projective transformation}. The automorphisms of $\R^{2}$ extend to projective transformations that fix the line at infinity. Nevertheless, in general, the line at infinity can be interchanged with another line trough a projective transformation. It follows that, if we take an ellipse and we apply a transformation in such a way that the new line at infinity is tangent to the ellipse, then we get a parabola; similarly, if the new line at infinity intersects the ellipse in two points, then we get a hyperbola. Hence, in the projective framework, we have essentially one irreducible conic.

\subsubsection{Weierstrass normal form}

Let us come back to the generic plane cubic curve with integral coefficients \eqref{DiophCubic}, that we call $\mathfrak{C}$. We denote by $\bar{\mathfrak{C}}$ its projective closure. We suppose that $O$ is a \emph{smooth rational} point of $\bar{\mathfrak{C}}$. In this case, there always exists a change of variables (see section \ref{EllipticCurvesSec} below for more details) that sends $O$ to the point at infinity of the $y$-axis (that is, the point at infinity corresponding to the vertical direction) and takes the affine curve $\mathfrak{C}$ to the following form, called \emph{Weierstrass normal form}:
\begin{equation}\label{WNForm}
	y^{2} = x^{3} + bx + c, \quad b, c \in \Z
\end{equation}
In fact, we argue in the following way, represented in figure \ref{StepToWeierstrass} (see \cite[Section 1.3]{ST}): we choose the line at infinity as the tangent line to $\bar{\mathfrak{C}}$ in $O$; we call $P$ the third intersection point between the line at infinity and $\bar{\mathfrak{C}}$ (which necessarily exists), and we choose as $y$-axis the tangent line to $\bar{\mathfrak{C}}$ in $P$; lastly, we choose as $x$-axis as a line \emph{not} at infinity through $O$. In equation \eqref{DiophCubic}, the third-degree component $ax^{3} + bx^{2}y + cxy^{2} + dy^{3}$ represents the intersection with the line at infinity. Since we have two intersections with $y = 0$ and one intersection with $x = 0$, only the term $cxy^{2}$ is present, that is, $a = b = d = 0$. Thus, if we set $x = 0$, we get $gy^{2} + iy + j$. In this case, the intersection with the line at infinity can be obtained by completing to a third-degree equation through a new variable $u$---that is, $gy^{2}u + iyu^{2} + ju^{3} = 0$---and setting $u = 0$. Since such intersection is double (in $P$) by construction, we have $g = 0$. Thus, equation \eqref{DiophCubic} becomes
	\[cxy^{2} + ex^{2} + fxy + hx + iy + j = 0
\]
with $c \neq 0$. We divide by $c$ and we get $xy^{2} + (f'x + i')y = e'x^{2} + h'x + j'$. We multiply both sides by $x$ and we set $\bar{x} := x$ and $\bar{y} := xy$. We get\footnote{Multiplying both sides by $x$, we add the line $x = 0$ to the curve, but the coordinate change $\bar{x} := x$ and $\bar{y} := xy$ shrinks that line to the origin.}
	\[\bar{y}^{2} + (f'\bar{x} + i')\bar{y} = e'\bar{x}^{3} + h'\bar{x}^{2} + j'\bar{x}.
\]
We complete the square on the l.h.s.\ as follows:
	\[\textstyle \bigl(\bar{y} + \frac{1}{2}(f'\bar{x} + i')\bigr)^{2} = e'\bar{x}^{3} + h'\bar{x}^{2} + j'\bar{x} + \frac{1}{4}(f'\bar{x} + i')^{2}
\]
Multiplying both sides by ${e'}^{2}$ and applying the substitution $\hat{y} := e'\bigl(\bar{y} + \frac{1}{2}(f'\bar{x} + i')\bigr)$ and $\hat{x} := e'\bar{x}$, we get
	\[\hat{y}^{2} = \hat{x}^{3} + \alpha\hat{x}^{2} + \beta\hat{x} + \gamma.
\]
Lastly, we set $y := \hat{y}$ and $x := \hat{x} + \alpha$, obtaining \eqref{WNForm}.
\begin{figure}
\begin{center}
\begin{tikzpicture}
\filldraw[thick, blue, fill=cyan!30!white] (0pt, 0pt) circle(50pt);
\draw[thick, black] (-50pt, 0pt) -- (50pt, 0pt);
\draw[thick, black] (0pt, -50pt) -- (0pt, 50pt);
\draw[thick, orange] (50pt, 0pt) arc(0:45:25pt);
\draw[thick, orange] (50pt, 0pt) arc(0:-45:25pt);
\draw[thick, orange] (0pt, 50pt) arc(0:-45:25pt);
\draw[thick, orange] (0pt, -50pt) arc(180:135:25pt);
\filldraw[orange] (-50pt, 0pt) circle (1.5pt) node[left]{$O$};
\filldraw[orange] (50pt, 0pt) circle (1.5pt) node[right]{$O$};
\filldraw[orange] (0pt, -50pt) circle (1.5pt) node[below]{$P$};
\filldraw[orange] (0pt, 50pt) circle (1.5pt) node[above]{$P$};
\end{tikzpicture}
\end{center}
\caption{} \label{StepToWeierstrass}
\end{figure}

This normal form has been obtained under the assumption that the curve admits a smooth rational point, but we are going to show that this is not a restrictive assumption for our purposes here. In fact, we have three possibilities for a generic plane cubic $\bar{\mathfrak{C}}$: it can be (i) reducible, (ii) irreducible and singular, or (iii) smooth. In case (i), we essentially come back to the case of lines and conics;\footnote{Again, we have to consider that a cubic with rational coefficients can factorize in lines or conics with some irrational coefficients, but we can deal with this issue quite easily.} case (ii) can be analysed with a technique similar to the one used about conics, as the next section shows; and case (iii) will be the heart of our discussion. 

\subsubsection{Singular irreducible cubics}

A singular cubic has only one singular point (since, otherwise, a line through two distinct singularities would intersect the cubic at least four times, a contradiction). In this case, one can prove that the singularity is necessarily a rational point, thus problem (P1) has always a positive answer. Moreover, since we are assuming $\bar{\mathfrak{C}}$ irreducible, it always has a smooth rational point. Indeed, we fix a line $\mathfrak{L}$ with rational coefficients that intersects $\bar{\mathfrak{C}}$ twice (not more) in the singularity. In this case, there exists a third intersection between $\bar{\mathfrak{C}}$ and $\mathfrak{L}$, which is a smooth rational point of $\bar{\mathfrak{C}}$.

It follows that we can express $\mathfrak{C}$ in the normal form \eqref{WNForm}. In this setting, $\bar{\mathfrak{C}}$ is singular if and only if the polynomial $p(x) = x^{3} + bx + c$ has at least one repeated (real) root. Thus, $y^{2} = (x - \alpha)^{2}(x - \beta)$ with $\alpha, \beta \in \R$. Up to a translation, we suppose $\alpha = 0$. We get $y^{2} = x^{2}(x - \beta)$ with $\beta \in \R$. There are three possible shapes for a curve in this family, corresponding to the cases $\beta < 0$, $\beta = 0$, and $\beta > 0$, represented in figure \ref{SingCubics}. The projective closures of these curves intersect the line at infinity only in the direction of the $y$-axis, therefore we get the shapes represented in figure \ref{SingCubicsProj}.

\begin{figure}
\begin{center}
\begin{tikzpicture}
\draw[->] (-60pt, 0pt) -- (60pt, 0pt);
\draw[->] (0pt, -60pt) -- (0pt, 60pt);
\draw[thick, blue] (0pt, 0pt) .. controls (15pt, 15pt) and (35pt, 40pt) .. (45pt, 60pt);
\draw[thick, blue] (0pt, 0pt) .. controls (15pt, -15pt) and (35pt, -40pt) .. (45pt, -60pt);
\draw[thick, blue] (0pt, 0pt) .. controls (-10pt, 15pt) and (-47.5pt, 20pt) .. (-50pt, 0pt);
\draw[thick, blue] (0pt, 0pt) .. controls (-10pt, -15pt) and (-47.5pt, -20pt) .. (-50pt, 0pt);
\filldraw[blue] (-50pt, 0pt) circle (1.5pt);
\filldraw[blue] (0pt, 0pt) circle (1.5pt);
\draw[blue] (-57pt, -10pt) node {$\beta$};

\begin{scope}[xshift=150pt]
\draw[->] (-60pt, 0pt) -- (60pt, 0pt);
\draw[->] (0pt, -60pt) -- (0pt, 60pt);
\draw[thick, blue] (0pt, 0pt) .. controls (15pt, 10pt) and (35pt, 40pt) .. (45pt, 60pt);
\draw[thick, blue] (0pt, 0pt) .. controls (15pt, -10pt) and (35pt, -40pt) .. (45pt, -60pt);
\filldraw[blue] (0pt, 0pt) circle (1.5pt);
\draw[blue] (-7pt, -10pt) node {$\beta$};
\end{scope}

\begin{scope}[xshift=300pt]
\draw[->] (-60pt, 0pt) -- (60pt, 0pt);
\draw[->] (0pt, -60pt) -- (0pt, 60pt);
\filldraw[blue] (0pt, 0pt) circle (1.5pt);
\draw[thick, blue] (40pt, 0pt) .. controls (42.5pt, 30pt) and (55pt, 40pt) .. (60pt, 60pt);
\draw[thick, blue] (40pt, 0pt) .. controls (42.5pt, -30pt) and (55pt, -40pt) .. (60pt, -60pt);
\filldraw[blue] (40pt, 0pt) circle (1.5pt);
\draw[blue] (33pt, -10pt) node {$\beta$};
\end{scope}

\end{tikzpicture}
\end{center}
\caption{} \label{SingCubics}
\end{figure}

\begin{figure}
\begin{center}
\begin{tikzpicture}
\filldraw[thick, blue, fill=cyan!30!white] (0pt, 0pt) circle(50pt);

\draw[thick, black] (0pt, -50pt) -- (0pt, 50pt);
\draw[thick, orange] (0pt, 0pt) .. controls (15pt, 10pt) and (15pt, 40pt) .. (0pt, 50pt);
\draw[thick, orange] (0pt, 0pt) .. controls (15pt, -10pt) and (15pt, -40pt) .. (0pt, -50pt);
\draw[thick, orange] (0pt, 0pt) .. controls (-5pt, 10pt) and (-30pt, 15pt) .. (-35pt, 0pt);
\draw[thick, orange] (0pt, 0pt) .. controls (-5pt, -10pt) and (-30pt, -15pt) .. (-35pt, 0pt);
\filldraw[orange] (0pt, 0pt) circle (1.5pt);
\filldraw[orange] (-35pt, 0pt) circle (1.5pt);
\filldraw[thick, black, fill=orange] (0pt, 50pt) circle (1.5pt);
\filldraw[thick, black, fill=orange] (0pt, -50pt) circle (1.5pt);

\begin{scope}[xshift=150pt]
\filldraw[thick, blue, fill=cyan!30!white] (0pt, 0pt) circle(50pt);

\draw[thick, black] (0pt, -50pt) -- (0pt, 50pt);
\draw[thick, orange] (0pt, 0pt) .. controls (15pt, 5pt) and (15pt, 40pt) .. (0pt, 50pt);
\draw[thick, orange] (0pt, 0pt) .. controls (15pt, -5pt) and (15pt, -40pt) .. (0pt, -50pt);
\filldraw[orange] (0pt, 0pt) circle (1.5pt);
\filldraw[thick, black, fill=orange] (0pt, 50pt) circle (1.5pt);
\filldraw[thick, black, fill=orange] (0pt, -50pt) circle (1.5pt);
\end{scope}

\begin{scope}[xshift=300pt]
\filldraw[thick, blue, fill=cyan!30!white] (0pt, 0pt) circle(50pt);

\draw[thick, black] (0pt, -50pt) -- (0pt, 50pt);
\draw[thick, orange] (15pt, 0pt) .. controls (17.5pt, 15pt) and (20pt, 20pt) .. (22.5pt, 30pt) .. controls (22.5pt, 40pt) and (10pt, 45pt) .. (0pt, 50pt);
\draw[thick, orange] (15pt, 0pt) .. controls (17.5pt, -15pt) and (20pt, -20pt) .. (22.5pt, -30pt) .. controls (22.5pt, -40pt) and (10pt, -45pt) .. (0pt, -50pt);
\filldraw[orange] (0pt, 0pt) circle (1.5pt);
\filldraw[orange] (15pt, 0pt) circle (1.5pt);
\filldraw[thick, black, fill=orange] (0pt, 50pt) circle (1.5pt);
\filldraw[thick, black, fill=orange] (0pt, -50pt) circle (1.5pt);
\end{scope}

\end{tikzpicture}
\end{center}
\caption{} \label{SingCubicsProj}
\end{figure}

In order to solve problem (P2), we argue similarly to the case of conics:
\begin{itemize}
	\item We call $\mathfrak{L}$ the vertical line $x = 1$.
	\item We get the projection $\pi \colon \mathfrak{C} \setminus \{(0, 0)\} \to \mathfrak{L}$ defined as follows: for every point $Q \in \mathfrak{C}$ different from the origin (that is, from the singularity), we consider the line passing trough the origin and $Q$; this line is not vertical, thus it intersects $\mathfrak{L}$ in a unique point $R$; we set $\pi(Q) := R$.
	\item This projection is bijective, except for the curve $y^{2} = x^{3}$, in which we have to remove $(1, 0)$ from $\mathfrak{L}$ because the $x$-axis intersects $\mathfrak{C}$ in the origin three times. The inverse is defined as follows: for every point $R$ in the image of $\pi$, we consider the line through the origin and $R$, that intersects the conic twice in the origin itself and another time in a point $Q$; we have $Q = \pi^{-1}(R)$.
	\item The point $Q$ is rational if and only if $\pi(Q)$ is rational. Thus, the rational points of $\mathfrak{C}$ are the ones of the form $\pi^{-1}(1, q)$ with $q \in \Q$.
\end{itemize}
We can easily compute the point $\pi^{-1}(1, q)$. In fact, this point belongs to the line passing through the origin and $(1, q)$, which is $y = qx$. Hence, we just have to impose the condition $y = qx$ within the equation of the curve. For example, if we consider $y^{2} = x^{3}$, by imposing $y = qx$ we get $(qx)^{2} = x^{3}$, that is, $x = q^{2}$, thus $y = q^{3}$. It follows that the rational points of this curve are the ones of the form $(q^{2}, q^{3})$ for any $q \in \Q$.

\subsubsection{Smooth cubics with a rational point}

In the case of a smooth plane cubic curve $\mathfrak{C}$, problem (P1) is still open, that is, we do not have a general technique to establish whether \eqref{DiophCubic} admits a rational point or not. Thus, in this framework, we assume that there exists a rational point (necessarily smooth) and we concentrate on problem (P2). Under this assumption, we can express $\mathfrak{C}$ in the Weierstrass normal form \eqref{WNForm}, with $4b^{2} + 27c^{2} \neq 0$. The polynomial $p(x) = x^{3} + bx + c$ has three distinct complex roots, therefore we have two possibilities: (i) $y^{2} = (x-\alpha)(x-\beta)(x-\gamma)$, where $\alpha, \beta, \gamma \in \R$ are distinct; and (ii) $y^{2} = (x-\alpha)(x-z)(x-\bar{z})$, where $\alpha \in \R$ and $z \in \C \setminus \R$. We get the two shapes shown in figure \ref{SmoothCubics}. Again, the corresponding projective closures intersect the line at infinity only in the direction of the $y$-axis, therefore we get the shapes represented in figure \ref{SmoothCubicsProj}.

In this case, we cannot project the curve (or the complement of a point) into a line. Indeed, given a point $P$ of the cubic and a point $R$ of a fixed line, the line through $P$ and $R$ intersects the cubic twice, while in the cases of conics and singular cubics we had only one intersection (and for this reason $\pi$ was injective). Hence, we have to analyse smooth cubics in a different way.

\begin{figure}
\begin{center}
\begin{tikzpicture}
\draw[->] (-60pt, 0pt) -- (60pt, 0pt);
\draw[->] (0pt, -60pt) -- (0pt, 60pt);

\draw[thick, blue] (10pt, 0pt) .. controls (7.5pt, 25pt) and (-37.5pt, 25pt) .. (-40pt, 0pt);
\draw[thick, blue] (10pt, 0pt) .. controls (7.5pt, -25pt) and (-37.5pt, -25pt) .. (-40pt, 0pt);
\filldraw[blue] (10pt, 0pt) circle (1.5pt);
\filldraw[blue] (-40pt, 0pt) circle (1.5pt);
\draw[blue] (-47pt, -10pt) node {$\alpha$};
\draw[blue] (17pt, -10pt) node {$\beta$};

\draw[thick, blue] (40pt, 0pt) .. controls (42.5pt, 30pt) and (55pt, 40pt) .. (60pt, 60pt);
\draw[thick, blue] (40pt, 0pt) .. controls (42.5pt, -30pt) and (55pt, -40pt) .. (60pt, -60pt);
\filldraw[blue] (40pt, 0pt) circle (1.5pt);
\draw[blue] (47pt, -9pt) node {$\gamma$};

\begin{scope}[xshift=200pt]
\draw[->] (-60pt, 0pt) -- (60pt, 0pt);
\draw[->] (0pt, -60pt) -- (0pt, 60pt);
\draw[thick, blue] (10pt, 0pt) .. controls (12.5pt, 20pt) and (40pt, 40pt) .. (50pt, 60pt);
\draw[thick, blue] (10pt, 0pt) .. controls (12.5pt, -20pt) and (40pt, -40pt) .. (50pt, -60pt);
\filldraw[blue] (10pt, 0pt) circle (1.5pt);
\draw[blue] (5pt, -9pt) node {$\alpha$};
\end{scope}

\end{tikzpicture}
\end{center}
\caption{} \label{SmoothCubics}
\end{figure}

\begin{figure}
\begin{center}
\begin{tikzpicture}
\filldraw[thick, blue, fill=cyan!30!white] (0pt, 0pt) circle(50pt);

\draw[thick, black] (0pt, -50pt) -- (0pt, 50pt);
\draw[thick, orange] (10pt, 0pt) .. controls (5pt, 15pt) and (-30pt, 15pt) .. (-35pt, 0pt);
\draw[thick, orange] (10pt, 0pt) .. controls (5pt, -15pt) and (-30pt, -15pt) .. (-35pt, 0pt);
\draw[thick, orange] (25pt, 0pt) .. controls (27.5pt, 12.5pt) and (30pt, 17.5pt) .. (32.5pt, 25pt) .. controls (32.5pt, 35pt) and (15pt, 45pt) .. (0pt, 50pt);
\draw[thick, orange] (25pt, 0pt) .. controls (27.5pt, -12.5pt) and (30pt, -17.5pt) .. (32.5pt, -25pt) .. controls (32.5pt, -35pt) and (15pt, -45pt) .. (0pt, -50pt);
\filldraw[orange] (10pt, 0pt) circle (1.5pt);
\filldraw[orange] (-35pt, 0pt) circle (1.5pt);
\filldraw[orange] (25pt, 0pt) circle (1.5pt);
\filldraw[thick, black, fill=orange] (0pt, 50pt) circle (1.5pt);
\filldraw[thick, black, fill=orange] (0pt, -50pt) circle (1.5pt);

\begin{scope}[xshift=200pt]
\filldraw[thick, blue, fill=cyan!30!white] (0pt, 0pt) circle(50pt);

\draw[thick, black] (0pt, -50pt) -- (0pt, 50pt);
\draw[thick, orange] (15pt, 0pt) .. controls (17.5pt, 15pt) and (20pt, 20pt) .. (22.5pt, 30pt) .. controls (22.5pt, 40pt) and (10pt, 45pt) .. (0pt, 50pt);
\draw[thick, orange] (15pt, 0pt) .. controls (17.5pt, -15pt) and (20pt, -20pt) .. (22.5pt, -30pt) .. controls (22.5pt, -40pt) and (10pt, -45pt) .. (0pt, -50pt);
\filldraw[orange] (15pt, 0pt) circle (1.5pt);
\filldraw[thick, black, fill=orange] (0pt, 50pt) circle (1.5pt);
\filldraw[thick, black, fill=orange] (0pt, -50pt) circle (1.5pt);
\end{scope}

\end{tikzpicture}
\end{center}
\caption{} \label{SmoothCubicsProj}
\end{figure}

\section{The Abelian Group Structure}

Let us consider a smooth plane cubic $\mathfrak{C}$ and its projective closure $\bar{\mathfrak{C}}$. Given two points $P, Q \in \bar{\mathfrak{C}}$, we define a third point $R$ as follows. If $P \neq Q$, then we fix the line $\mathfrak{L}$ passing through them; otherwise, we call $\mathfrak{L}$ the tangent line to $\bar{\mathfrak{C}}$ in $P$, that is well-defined since $\bar{\mathfrak{C}}$ is smooth. The line $\mathfrak{L}$ intersects $\bar{\mathfrak{C}}$ in another point $R$. This would be false in general in the affine plane $\R^{2}$ (for example, the vertical line $x = \frac{1}{2}$ intersects the smooth cubic $y^{2} = x(x-1)(x-2)$ only in two distinct points), but it is true in the projective framework (in the previous example, the third intersection point is the point at infinity of the $y$-axis).

We observe that, if $P \neq Q$, then we have two possibilities: if $\mathfrak{L}$ is tangent to $\bar{\mathfrak{C}}$ in $P$ or $Q$, then $R$ coincides respectively with $P$ or $Q$; otherwise, it is a third distinct point. Similarly, if $P = Q$, then we have two possibilities: if $\mathfrak{L}$ intersects three times $\bar{\mathfrak{C}}$ in $P$, then $P = Q = R$; otherwise, $P = Q \neq R$.

In this way, we get the operation $* \, \colon \bar{\mathfrak{C}} \times \bar{\mathfrak{C}} \to \bar{\mathfrak{C}}$, $(P, Q) \mapsto R$. Thus, we denote the point $R$ by $P * Q$. This operation is clearly commutative, but not associative in general. Nevertheless, we can use it to construct a more interesting operation, defined as follows: we fix arbitrarily a point $O \in \bar{\mathfrak{C}}$ and we set $P + Q := O * (P * Q)$. We get the operation $+ \, \colon \bar{\mathfrak{C}} \times \bar{\mathfrak{C}} \to \bar{\mathfrak{C}}$, represented in figure \ref{SumCubic}, that induces an \emph{abelian group} structure on $\bar{\mathfrak{C}}$.

\begin{figure}
\begin{center}
\begin{tikzpicture}
\draw[->] (-60pt, 0pt) -- (60pt, 0pt);
\draw[->] (0pt, -60pt) -- (0pt, 60pt);

\draw[thick, blue] (10pt, 0pt) .. controls (7.5pt, 25pt) and (-37.5pt, 25pt) .. (-40pt, 0pt);
\draw[thick, blue] (10pt, 0pt) .. controls (7.5pt, -25pt) and (-37.5pt, -25pt) .. (-40pt, 0pt);
\draw[thick, blue] (40pt, 0pt) .. controls (42.5pt, 30pt) and (55pt, 40pt) .. (60pt, 60pt);
\draw[thick, blue] (40pt, 0pt) .. controls (42.5pt, -30pt) and (55pt, -40pt) .. (60pt, -60pt);

\draw[orange] (-60pt, -25pt) -- (60pt, 40pt);
\draw[magenta] (-60pt, 10.5pt) -- (60pt, 10.5pt);

\filldraw[orange] (-35pt, -11.5pt) circle (1.5pt);
\filldraw[orange] (49.5pt, 34.5pt) circle (1.5pt);
\filldraw[magenta] (5.5pt, 10.5pt) circle (1.5pt);
\filldraw[magenta] (-35.5pt, 10.5pt) circle (1.5pt);
\filldraw[green] (41.5pt, 10.5pt) circle (1.5pt);

\draw[orange] (-36pt, -18.5pt) node {$\scriptstyle P$};
\draw[orange] (47pt, 41pt) node {$\scriptstyle Q$};
\draw[magenta] (-38pt, 16pt) node {$\scriptstyle O$};
\draw[green] (55pt, 16pt) node {$\scriptstyle P+Q$};
\end{tikzpicture}
\end{center}
\caption{} \label{SumCubic}
\end{figure}

It is not complicated to prove that this sum is associative, but we omit the details here. The fixed point $O$ turns out to be the zero of the group. In fact, by definition we have $O + P = O * (O * P)$. The line $\mathfrak{L}$ through $O$ and $P$ intersects the cubic in the point $Q := O * P$. Now we have to consider the line through $O$ and $Q$, that is $\mathfrak{L}$ itself. Thus, the third intersection point, beyond $O$ and $Q$, is $P$. This shows that $O + P = P$. Similarly, let us show that $-P = P * (O * O)$. Indeed, let us compute $P + (-P) = O * (P * (-P)) = O * (P * (P * (O * O)))$. We call $\mathfrak{L}$ the tangent line to $\bar{\mathfrak{C}}$ in $O$ and we call $Q$ the third intersection point between $\mathfrak{L}$ and $\bar{\mathfrak{C}}$. Moreover, we call $\mathfrak{L}'$ the line through $P$ and $Q$ and we already called $-P$ the the third intersection point between $\mathfrak{L}'$ and $\bar{\mathfrak{C}}$. Now we have to consider the line through $P$ and $-P$, that is $\mathfrak{L}'$ itself. Thus, the third intersection point is $Q$. Lastly, we have to consider the line through $O$ and $Q$, that is $\mathfrak{L}$. It follows that the third intersection point is $O$ (since $\mathfrak{L}$ is tangent there), thus $P + (-P) = O$.

When the cubic is expressed in the Weierstrass normal form \eqref{WNForm}, it is common to fix $O$ as the point at infinity of the $y$-axis. In this case, in order to compute $P + Q$, that is, $O * (P * Q)$, we consider the vertical line through $P * Q$ and we take the other intersection point with $\bar{\mathfrak{C}}$. Such intersection is the reflection of $P * Q$ with respect to the $x$-axis, as figure \ref{SumCubicWeierstress} shows.

\begin{figure}
\begin{center}
\begin{tikzpicture}
\draw[->] (-60pt, 0pt) -- (60pt, 0pt);
\draw[->] (0pt, -60pt) -- (0pt, 60pt);

\draw[thick, blue] (10pt, 0pt) .. controls (7.5pt, 25pt) and (-37.5pt, 25pt) .. (-40pt, 0pt);
\draw[thick, blue] (10pt, 0pt) .. controls (7.5pt, -25pt) and (-37.5pt, -25pt) .. (-40pt, 0pt);
\draw[thick, blue] (40pt, 0pt) .. controls (42.5pt, 30pt) and (55pt, 40pt) .. (60pt, 60pt);
\draw[thick, blue] (40pt, 0pt) .. controls (42.5pt, -30pt) and (55pt, -40pt) .. (60pt, -60pt);

\draw[orange] (-60pt, -25pt) -- (60pt, 40pt);
\draw[magenta] (5.5pt, 60pt) -- (5.5pt, -60pt);

\filldraw[orange] (-35pt, -11.5pt) circle (1.5pt);
\filldraw[orange] (49.5pt, 34.5pt) circle (1.5pt);
\filldraw[magenta] (5.5pt, 10.5pt) circle (1.5pt);
\filldraw[green] (5.5pt, -10.5pt) circle (1.5pt);

\draw[orange] (-36pt, -18.5pt) node {$\scriptstyle P$};
\draw[orange] (47pt, 41pt) node {$\scriptstyle Q$};
\draw[green] (20pt, -10.5pt) node {$\scriptstyle P+Q$};
\end{tikzpicture}
\end{center}
\caption{} \label{SumCubicWeierstress}
\end{figure}

Now we suppose that $P$ and $Q$ are rational points of $\bar{\mathfrak{C}}$. In this case, the line $\mathfrak{L}$ through $P$ and $Q$ is rational, and it is quite easy to prove that the third intersection point between $\mathfrak{L}$ and $\bar{\mathfrak{C}}$ is rational too. Thus, if the subset of rational points of $\bar{\mathfrak{C}}$ is not empty, then it forms a subgroup, that we denote by $(\bar{\mathfrak{C}}_{\,\Q}, +)$. The following fundamental theorem holds.

\begin{Theorem}[Mordell] For any smooth rational plane cubic $\mathfrak{C}$, if the subset of rational points of $\bar{\mathfrak{C}}$ is not empty, then the abelian group $(\bar{\mathfrak{C}}_{\,\Q}, +)$ is finitely generated.
\end{Theorem}

The proof is not trivial and it can be found in \cite[Section 3.5]{ST}. Thanks to this theorem, we have $\bar{\mathfrak{C}}_{\,\Q} \simeq \Z^{\oplus r} \oplus \Tor\,\bar{\mathfrak{C}}_{\,\Q}$, and the number $r$ is called \emph{rank} of the group.

\begin{Def} The \emph{algebraic rank} of a smooth rational plane cubic $\mathfrak{C}$, with at least one rational point, is the rank of the abelian group $(\bar{\mathfrak{C}}_{\,\Q}, +)$.
\end{Def}

There are known techniques to compute the torsion subgroup. In particular, Mazur's theorem states that is can be $\Z_{n}$, with $1 \leq n \leq 10$ or $n = 12$, or $\Z_{2} \oplus \Z_{2n}$, with $1 \leq n \leq 4$. On the contrary, we do not know how to compute the rank in general. We neither know whether the set of possible ranks is bounded above or not, in spite of a rich and deep research going on about this topic in the last decades.

%%%%%%%%%%%%%%%%%%%%%%%%%%%%%%%%%%%%%%%%%%%%%%%%%%%%%%%%%%%%%%%%%%%%%%%%%%%%%%%%%%%%%%%%%%%%%%%

\section{Elliptic Curves}\label{EllipticCurvesSec}

We are considering rational plane cubic curves whose projective closure $\bar{\mathfrak{C}}$ is smooth and has a fixed rational point $O$. The pair $(\bar{\mathfrak{C}}, O)$ is an \emph{elliptic curve}, as the following temporary definition states.

\begin{Def}[Temporary]\label{DefEllCurveTemp} An \emph{elliptic curve} is a pair $(\mathfrak{P}, O)$, where $\mathfrak{P}$ is a rational smooth cubic curve in the projective plane and $O \in \mathfrak{P}$ is a fixed rational point.
\end{Def}

Since, up to a projective transformation, every projective curve is the closure of a curve in $\R^{2}$, the case $(\bar{\mathfrak{C}}, O)$ considered above is essentially the generic one. Moreover, since the curve is smooth, the topological genus $g$ is related to the degree $d$ by the formula $g = \frac{1}{2}(d-1)(d-2)$; therefore, $d = 3$ if and only if $g = 1$. For this reason, we can replace the expression ``cubic curve'' in definition \ref{DefEllCurveTemp} with the expression ``curve of genus $1$''.

This definition is not satisfactory yet, nor it explains the reason of the word ``elliptic'', since the ellipse is definitely not a cubic. We can find a better definition as follows. First of all, it is not necessary to restrict to plane cubics: we can consider any rational projective curve (embedded in a projective space of any dimension) that is isomorphic to a plane cubic. This is a step forward, but it is not enough yet. Indeed, we can consider a weaker notion of equivalence than isomorphism, that is, \emph{birational equivalence}. Essentially, two curves are birationally equivalent when an open dense subset of the first curve is isomorphic to an open dense subset of the second one. We already saw some examples above: the complement of a point in a circle is isomorphic to a line, thus the circle is birationally equivalent to a line. A birational equivalence can be expressed through a rational function, that is, a quotient of polynomials. Here we are interested in birational equivalences over $\Q$, since they respect rational points. Thus, we give the following definition.

\begin{Def}\label{DefEllCurve} An \emph{elliptic curve} is a pair $(\mathfrak{P}, O)$, where $\mathfrak{P}$ is a rational projective curve that is birationally equivalent (over $\Q$) to a smooth cubic curve in the projective plane, and $O \in \mathfrak{P}$ is a fixed rational point.
\end{Def}

With this definition, a quartic curve of the form
\begin{equation}\label{QuarticElliptic}
	y^{2} = ax^{4} + bx^{3} + cx^{2} + dx + e, \quad a \neq 0
\end{equation}
with a fixed rational point is an elliptic curve if the polynomial $p(x) = ax^{4} + bx^{3} + cx^{2} + dx + e$ has no repeated roots. For example, the quartic $y^{2} = 1 - x^{4}$ is birationally equivalent to the smooth cubic $\tilde{y}^{2} = \tilde{x}^{3} - \frac{3}{8}\tilde{x}^{2} + \frac{1}{16}\tilde{x} - \frac{1}{256}$ through the equivalence $\tilde{x} := -\frac{1}{4} \frac{1}{x - 1}$ and $\tilde{y} := \frac{1}{16}\frac{y}{(x-1)^{2}}$. In general, we can always find a rational equivalence of this kind.

Now we can justify the name ``elliptic'' for these curves. Let us consider the integral providing the perimeter of the ellipse $\frac{x^{2}}{a^{2}} + \frac{y^{2}}{b^{2}} = 1$. Computing the arc in the quadrant $x, y \geq 0$, we get $y = \frac{b}{a}\sqrt{a^{2} - x^{2}}$, thus $y' = -\frac{b}{a}\frac{x}{\sqrt{a^{2} - x^{2}}}$. It follows that the perimeter is given by
\begin{equation}\label{PerimEllipse}
	4\int_{0}^{a} \sqrt{1 + (y')^{2}} \, dx = 4\int_{0}^{a} \sqrt{1 + \frac{b^{2}}{a^{2}} \frac{x^{2}}{a^{2} - x^{2}}} \, dx = 4a \int_{0}^{1} \sqrt{\frac{1-k^{2}t^{2}}{1 - t^{2}}} \, dt,
\end{equation}
where, in the last equality, we set $t := \frac{x}{a}$ and $k := \sqrt{1 - \bigl(\frac{b}{a}\bigr)^{2}}$. Of course, the perimeter of an arc can be obtained by integrating until a fixed value $t_{0} \in [0, 1]$. Let us analyse the function to be integrated, that is, $u = \sqrt{\frac{1-k^{2}t^{2}}{1 - t^{2}}}$. We write it in the form $y = \sqrt{\frac{1-k^{2}x^{2}}{1 - x^{2}}}$. We get $y^{2}(1 - x^{2}) = 1 - k^{2}x^{2}$. Now we consider the birational equivalence $\tilde{x} := x$ and $\tilde{y} := \frac{y}{1 - x^{2}}$. We obtain $\tilde{y}^{2} = (1 - \tilde{x}^{2})(1 - k^{2}\tilde{x}^{2})$, with $k \neq 1$ (since, otherwise, $b = 0$). This is (a part of) a quartic of the form \eqref{QuarticElliptic}, which is an elliptic curve by definition.

We observe that the integral \eqref{PerimEllipse} can be written in the form
	\[4a \int_{0}^{1} \frac{1-k^{2}t^{2}}{\sqrt{(1-k^{2}t^{2})(1 - t^{2})}} \, dt.
\]
In general, we call \emph{elliptic integral} an integral of the form $\int R(x, \sqrt{\varphi(x)}) dx$, where $R$ is a rational function and $\varphi$ is a polynomial of degree $3$ or $4$ with no multiple roots (see \cite{Takebe}). These integrals have been deeply studied because of many interesting applications (arclength of a lemniscate, period of a pendulum, etc.) and a computation similar the one above shows that the integrand function is (a part of) an elliptic curve. This is the reason of the name: elliptic curves are so called because they naturally arise from elliptic integrals, which are called in this way since the first meaningful example is the integral providing the perimeter of an ellipse.

The theory we are briefly introducing here concerns elliptic curves. Actually, we should prove that the abelian group structure constructed above and the tools we will discuss below are invariant under birational equivalences (this is also necessary to put every cubic in Weierstrass normal form without loss of generality). We do now show the details here, but it turns out that this invariance holds.

%%%%%%%%%%%%%%%%%%%%%%%%%%%%%%%%%%%%%%%%%%%%%%%%%%%%%%%%%%%%%%%%%%%%%%%%%%%%%%%%%%%%%%%%%%%%%%%

\section{$L$-Series of an Elliptic Curve}

Let us consider a smooth rational plane cubic $\mathfrak{C}$ expressed in the Weierstrass normal form \eqref{WNForm}. For every prime number $p$, we can project the coefficients to $\Z_{p}$; we get a curve in $\Z_{p}^{2}$, that we call \emph{reduction modulo $p$} of \eqref{WNForm} and we denote by $\mathfrak{C}_{p}$. The projective closure $\bar{\mathfrak{C}}_{p}$, that is well-defined too, is not necessarily smooth, but there exist only finitely many primes $p$ such that $\bar{\mathfrak{C}}_{p}$ is singular, that is, the prime divisors of the discriminant $\Delta := 4b^{3} + 27c^{2}$. In these cases, the corrections to be made are not particularly difficult.

How many points does the smooth curve $\bar{\mathfrak{C}}_{p}$ contain? It is not possible to reply through a simple function of $p$ and of the coefficients of $\mathfrak{C}$, but we can try to estimate this number up to a reasonable error, starting from the following simple lemma.

\begin{Lemma}\label{LemmaHalfSq} In the multiplicative group $\Z_{p}^{*}$, with $p > 2$, half of the elements are squares (that is, there are $\frac{p-1}{2}$ squares).
\end{Lemma}
\begin{proof} The subgroup of squares is the image of the group morphism $\varphi \colon \Z_{p}^{*} \to \Z_{p}^{*}$, $a \mapsto a^{2}$, hence it is isomorphic to $\frac{\Z_{p}^{*}}{\Ker\,\varphi}$. Since $\Ker\,\varphi$ contains two elements ($1$ and $-1$), the thesis immediately follows.
\end{proof}

Let us consider the equation of $\mathfrak{C}_{p}$, that is, $y^{2} = x^{3} + ax^{2} + bx + c$ in $\Z_{p}$. We set $p(x) := x^{3} + ax^{2} + bx + c$. Fixing $x_{0} \in \Z_{p}$, we have two possibilities:
\begin{itemize}
	\item If $x_{0}$ is a root of $p$, then $(0, x_{0})$ is the only point of $\mathfrak{C}_{p}$ of the form $(y, x_{0})$.
	\item If $x_{0}$ is not a root of $p$, then, varying $p$, the probability that $p(x_{0})$ is a square is $\frac{1}{2}$ because of lemma \ref{LemmaHalfSq}. If $p(x_{0}) = x_{1}^{2}$, then we have the two points $(\pm x_{1}, x_{0})$ in $\mathfrak{C}_{p}$, otherwise we have no points of the form $(y, x_{0})$. Hence, in average, we have one point for each $x_{0}$ that is not a root of $p$.
\end{itemize}
It follows that, in average, we have one point of $\mathfrak{C}_{p}$ of the form $(y, x_{0})$ for every $x_{0} \in \Z_{p}$, thus we expect that $\mathfrak{C}_{p}$ contains $p$ points. By adding the point at infinity of the $y$-axis, we get $p+1$ points for $\bar{\mathfrak{C}}_{p}$. Because of this argument, we express the cardinality of $\bar{\mathfrak{C}}_{p}$ as follows:
\begin{equation}\label{DefEpsilonP}
	\abs{\bar{\mathfrak{C}}_{p}} = p + 1 - \epsilon_{p},
\end{equation}
where $\epsilon_{p}$ is expected not to be ``too big'' (the minus sign is only a convention). In fact, the following result holds.

\begin{Theorem}[Hasse-Weyl]\label{TheoremHW} The term $\epsilon_{p}$ defined by \eqref{DefEpsilonP} satisfies the inequality
	\[\abs{\epsilon_{p}} \leq 2\sqrt{p}.
\]
\end{Theorem}

The set $\{\epsilon_{p}\}_{p \textnormal{ prime}}$ is a meaningful piece of information about the cubic $\mathfrak{C}$. In general, when we have a sequence of complex numbers $\{a_{n}\}_{n \in \N^{*}}$, it is interesting to construct the corresponding \emph{Dirichelet series}, that is, the following analytic function:
	\[f(s) = \sum_{n=1}^{+\infty} \frac{a_{n}}{n^{s}}
\]
A series of this form is useful because it has good additive and multiplicative properties. In many relevant cases, the function $f$ converges in an open subset of the complex plane and it can be extended by analytic continuation to a larger domain. The most famous example is the Riemann zeta function, that is, the Dirichelet series associated to the constant sequence $a_{n} = 1$. We get the series
\begin{equation}\label{RiemannZeta}
	\zeta(s) = \sum_{n=1}^{+\infty} \frac{1}{n^{s}},
\end{equation}
that converges in the domain $\textnormal{Re} \, s > 1$ and can be extended analytically to $\C \setminus \{1\}$. When the sum converges, it is equal to
\begin{equation}\label{RiemannZetaProd}
	\zeta(s) = \prod_{p \textnormal{ prime}} \Bigl(1 - \frac{1}{p^{s}}\Bigr)^{-1},
\end{equation}
and many interesting properties of this function follow from the equality between the two expressions \eqref{RiemannZeta} and \eqref{RiemannZetaProd}.

In the case of a smooth plane cubic, we aim to construct a Dirichelet series from the set $\{\epsilon_{p}\}_{p \textnormal{ prime}}$. Since we know the value of $\epsilon_{p}$ only for $p$ prime, we start from a suitable product similar to \eqref{RiemannZetaProd}. We define the \emph{L-function} of $\mathfrak{C}$ as:
\begin{equation}\label{LFuncProd}
	L_{\mathfrak{C}}(s) := \prod_{p \nmid \Delta} \Bigl(1 - \frac{\epsilon_{p}}{p^{s}} + \frac{1}{p^{2s-1}}\Bigr)^{-1} \cdot \prod_{p \mid \Delta} \Bigl(1 - \frac{\epsilon_{p}}{p^{s}}\Bigr)^{-1}
\end{equation}
Through quite simple computations, one can prove that this product coincides with the Dirichelet series
\begin{equation}\label{LFunc}
	L_{\mathfrak{C}}(s) = \sum_{n=1}^{+\infty} \frac{\epsilon_{n}}{n^{s}},
\end{equation}
where:
\begin{itemize}
	\item if $n = p$ is prime, then $\epsilon_{p}$ in formula \eqref{LFunc} coincides with the one in formulas \eqref{DefEpsilonP} and \eqref{LFuncProd};
	\item if $n = p^{k}$ is a prime power, then the recursive formula $\epsilon_{p^{k+1}} = \epsilon_{p^{k}}\epsilon_{p} - p\epsilon_{p^{k-1}}$ holds;
	\item if $n = p_{1}^{k_{1}} \cdots p_{h}^{k_{h}}$, where $p_{1}, \ldots, p_{h}$ are distinct primes, then $\epsilon_{n} = \epsilon_{p_{1}^{k_{1}}} \cdots \epsilon_{p_{h}^{k_{h}}}$.
\end{itemize}
It follows that the sequence $\{\epsilon_{n}\}_{n \in \N^{*}}$ is an extension of the previously defined sequence $\{\epsilon_{p}\}_{p \textnormal{ prime}}$, and it is a quite natural extension (except for the term $p\epsilon_{p^{k-1}}$ in the recursive formula, which is a little unexpected).

As a consequence of theorem \ref{TheoremHW}, one can prove that the series \eqref{LFunc} converges in the domain $\textnormal{Re} \, s > \frac{3}{2}$. Moreover, it can be extended to a holomorphic function defined in the whole complex plane as follows (see \cite{Liu}). For each prime $p$, the reduction modulo $p$ of $\mathfrak{C}$ can be: (i) smooth; (ii) singular with a \emph{node}, that is, with two distinct tangent lines in the singularity; (ii) singular with a \emph{cusp}, that is, with one tangent line at the singularity. We define $\varphi(p)$ as $0$ in case (i), $1$ in case (ii), and $2$ in case (iii); actually, if $p = 2$ or $3$, then case (iii) require a suitable correction term, which we do not analyse here. We set $N_{\mathfrak{C}} := \prod_{p} p^{\varphi(p)}$; this number, called \emph{conductor} of $\mathfrak{C}$, encodes the kinds of singular reductions of the curve at each prime divisor of the discriminant. Now we define the \emph{completed L-function} as
\begin{equation}\label{ExtendedLambda}
	\Lambda_{\mathfrak{C}}(s) := N^{\frac{s}{2}} (2\pi)^{-s} \Gamma(s) L_{\mathfrak{C}}(s),
\end{equation}
where $\Gamma$ is the usual Gamma function. It is quite easy to express \eqref{ExtendedLambda} through suitable integrals that converge to an entire function. Moreover, we have $\Lambda_{\mathfrak{C}}(s) = \epsilon_{\mathfrak{C}} \cdot \Lambda_{\mathfrak{C}}(2-s)$, where $\epsilon_{\mathfrak{C}} = \pm 1$.

From \eqref{ExtendedLambda}, we easily get $L_{\mathfrak{C}} \colon \C \to \C$ as an entire function. This justifies the following definition.

\begin{Def} The \emph{analytic rank} of a smooth rational plane cubic $\mathfrak{C}$, with at least one rational point, is the order of the holomorphic function $L_{\mathfrak{C}}$ in the point $s = 1$.
\end{Def}

Thus, we say the analytic rank of $\mathfrak{C}$ is $k \in \N$ if, in a neighbourhood of $1$ in $\C$, we have $L_{\mathfrak{C}}(s) = \sum_{n = k}^{+\infty} a_{n}(s-1)^{n}$, where $a_{n} \in \C$ for every $n \geq k$ and $a_{k} \neq 0$. Equivalently, $L_{\mathfrak{C}}(s) = (s-1)^{k}f(s-1)$, where $f$ is a holomorphic function such that $f(0) \neq 0$. It follows that the analytic rank is $0$ if and only if $L_{\mathfrak{C}}(1) \neq 0$, and the higher is the rank, the faster $L_{\mathfrak{C}}(s)$ tends to $0$ when $s$ tends to $1$.

%%%%%%%%%%%%%%%%%%%%%%%%%%%%%%%%%%%%%%%%%%%%%%%%%%%%%%%%%%%%%%%%%%%%%%%%%%%%%%%%%%%%%%%%%%%%%%%

\section{The Conjecture}

We defined the algebraic and the analytic rank of a smooth rational plane cubic. Now we can formulate the conjecture that constitute the main topic of this review.

\begin{BSDConj} For every smooth rational plane cubic $\mathfrak{C}$, with at least one rational point, the algebraic and the analytic rank coincide.
\end{BSDConj}

Let us try to understand why this conjecture was formulated. We know that the function $L_{\mathfrak{C}}(s)$ can be expressed through the product \eqref{LFuncProd} only when $\textnormal{Re} \, s > \frac{3}{2}$, but let us pretend for a moment that this is possible in a neighbourhood of $s = 1$ too. We can safely ignore the finitely many divisors of $\Delta$. We get:
\begin{equation}\label{WrongProduct}
	L_{\mathfrak{C}}(1) \, \textnormal{``}\!=\!\textnormal{''} \prod_{p \textnormal{ prime}} \Bigl(1 - \frac{\epsilon_{p}}{p} + \frac{1}{p}\Bigr)^{-1} = \prod_{p \textnormal{ prime}} \frac{p}{p - \epsilon_{p} + 1} \,\overset{\eqref{DefEpsilonP}}= \prod_{p \textnormal{ prime}} \frac{p}{\abs{\bar{\mathfrak{C}}_{p}}}
\end{equation}
Around 1958, Birch and Swinnerton-Dyer, two young researchers at Cambridge University, could access ones of the most powerful computers of that time. Making plenty of computations, they found ``experimentally'' that, when the algebraic rank of a smooth cubic increases, the cardinality $\abs{\bar{\mathfrak{C}}_{p}}$ tends to become much bigger than the expected average $p+1$, therefore the quotient $\frac{p}{\abs{\bar{\mathfrak{C}}_{p}}}$ becomes very small. A heuristic justification for this observation can be the following one: when the algebraic rank increases, we have plenty of rational points in $\bar{\mathfrak{C}}$, whose reduction modulo $p$ (when it is defined) produces plenty of points in $\bar{\mathfrak{C}}_{p}$. Since the quotient $\frac{p}{\abs{\bar{\mathfrak{C}}_{p}}}$ is the general term of the product \eqref{WrongProduct}, we argue as follows:
\begin{itemize}
	\item If $L_{\mathfrak{C}}(1) \neq 0$, then we are assuming that \eqref{WrongProduct} converges to a non-zero number, thus $\frac{p}{\abs{\bar{\mathfrak{C}}_{p}}} \to 1$ for $p \to +\infty$. Hence, the cardinality of $\bar{\mathfrak{C}}_{p}$ tends to be similar to the average $p+1$. This happens when $\mathfrak{C}$ has a finite number of rational points, since, otherwise, infinite rational points produce plenty of points in $\bar{\mathfrak{C}}_{p}$. Thus, the analytic rank is $0$ if and only if the algebraic rank is $0$.
	\item Assuming both ranks to be positive, if the algebraic rank is a high number, then we have plenty of rational points in $\bar{\mathfrak{C}}$, thus we have plenty of points in $\bar{\mathfrak{C}}_{p}$; therefore, $\frac{p}{\abs{\bar{\mathfrak{C}}_{p}}} \to 0$ quite fast when $p \to +\infty$. By continuity, if $s$ belongs to a sufficiently small neighbourhood $U$ of $1$, then the terms of the product \eqref{LFuncProd} tend to $0$ quite fast for $p \to \infty$, making $L_{\mathfrak{C}}(s)$ very small if $s \in U$. This means that the order of $L_{\mathfrak{C}}$ in $1$ is high.
\end{itemize}
Summarizing, it seems that the algebraic rank vanishes if and only if the analytic rank vanishes and, varying the curve, the algebraic rank increases if and only if the analytic rank increases. We repeat that this is a purely heuristic justification, since it is a quite imprecise argument, based on the wrong equality \eqref{WrongProduct}. Nevertheless, it allows us at least to hope that, by using the correct definition of $L_{\mathfrak{C}}$ through analytic continuation, the conjecture can be proved rigorously.

\SkipTocEntry \subsection{Current Status}

Up to now we are very far from achieving a proof of this conjecture, while we have a reasonably strong numerical evidence that it is true. The results proved up to now are the following ones:
\begin{enumerate}
	\item If a smooth plane cubic has analytic rank $0$ or $1$, then its algebraic rank is respectively $0$ or $1$.
	\item By introducing a suitable measure on the set of smooth plane cubics, Bhargava and Shankar proved that the subset of curves with analytic rank $0$ has positive measure. Hence, the previous item implies that there exists a subset with positive measure formed by curves that satisfy the conjecture.
\end{enumerate}
About the second item, the result stated is quite weak, since the measure of the subset might be very small. It has been proven again by Bhargava and Shankar that the average algebraic rank of all elliptic curves over $\Q$ is less than 1.17; this implies that at least 62.5\% of smooth plane cubics has algebraic rank $0$ or $1$, but we cannot draw any conclusion from this datum, since we only know the implication in item (1), which starts from the analytic side. Actually, we have some partial results about the converse implication in these cases, but under quite strong hypotheses. Therefore, in spite of various decades of high-level research, the conjecture is still widely open.

\SkipTocEntry \subsection{Applications}

Proving the conjecture would be very interesting for many reasons. First of all, we have effective algorithms to compute the analytic rank when it is at most $3$, thus, if the conjecture is true, we get the algebraic rank of these curves too.

Moreover, there are some meaningful problems that can be solved by applying the conjecture. For example, an old problem in number theory is the \emph{congruent number problem} (see \cite{Wiles}). We can state it by asking the following question: for which integral numbers $n$ there exists a right-angled triangle with rational length sides and area equal to $n$? We call $n$ \emph{congruent number} is it satisfies this property. Fibonacci showed that $5$ is congruent and Fermat proved that $1$ is not. One can prove that $n$ is congruent if and only if the smooth plane cubic $\mathfrak{C}_{n} \colon y^{2} = x^{3} - n^{2}x$ has infinite rational points. Assuming the Birch and Swinnerton-Dyer conjecture, this is equivalent to state that the analytic rank of $\mathfrak{C}_{n}$ is positive. In this case, it is possible to prove that every number that is equivalent to $5$, $6$, or $7$ modulo $8$ is congruent and it is possible to find an explicit criterion to establish whether any square-free integer is congruent or not.

For these reasons, and others that we do not mention here, proving the Birch and Swinnerton-Dyer Conjecture would be a breakthrough in contemporary mathematics.

%%%%%%%%%%%%%%%%%%%%%%%%%%%%%%%%%%%%%%%%%%%%%%%%%%%%%%%%%%%%%%%%%%%%%%%%%%%%%%%%%%%%%%%%%%%%%%%

\end{document}